\documentclass{amsart}

\usepackage{amssymb}
\usepackage{mathrsfs}
\usepackage{amsmath}
\usepackage{amscd}
\usepackage{graphicx}
\usepackage{pinlabel}
\usepackage{enumitem}

\usepackage{cite}

\newtheorem{theorem}{Theorem}
\newtheorem{prop}[theorem]{Proposition}
\newtheorem{lemma}[theorem]{Lemma}
\newtheorem{corollary}[theorem]{Corollary}

\newtheorem{rmk}[theorem]{Remark}

\def\R{\mathbb{R}}

\def\Z{\mathbb{Z}}

\begin{document}
\title[Flows of piecewise analytic vector fields]
{Flows of piecewise analytic vector fields in convex polytope decompositions}

\renewcommand{\theequation}{\arabic{section}.\arabic{subsection}.\arabic{equation}}
\numberwithin{equation}{section}
\numberwithin{theorem}{section}

\author{Tianqi Wu}
\maketitle

\begin{abstract}
We prove that for a convex polytope decomposition of a domain in $\mathbb R^n$, an integral curve of a piecewise analytic vector field is chopped by the decomposition into finitely many pieces. As a consequence, we prove the finiteness of the number of the edge flips in a discrete Yamabe flow.

\end{abstract}

\tableofcontents

\newpage
\section{Introduction}$ $
Gu et al. \cite{gu2018discrete} introduced a discrete Yamabe flow with edge flips, which are combinatorial changes of triangulations. It would be theoretically interesting and practically useful to know if there are only finitely many edge flips during the flow.

In this paper we prove that for a convex polytope decomposition of a domain in $\mathbb R^n$, an integral curve of a piecewise analytic vector field is chopped by the decomposition into finitely many pieces. As a consequence, we prove the finiteness of the edge flips in a discrete Yamabe flow mentioned above.

\subsection{Setup and the main theorem}

(a)
A subset $H$ of $\mathbb R^n$ is called a \emph{half space} if 
$$
H=\{x\in\mathbb R^n:a_1x_1+...+a_nx_n\geq b\}
$$
for some nonzero $(a_1,...,a_n)\in\mathbb R^n$ and $b\in\mathbb R$.
A subset $D$ of $\mathbb R^n$ is called a \emph{convex polytope} if $D$ has nonempty interior and is the intersection of finitely many half spaces.

(b)
Given $U\subseteq\mathbb R^n$, a \emph{vector field} on $U$ is a function from $U$ to $\mathbb R^n$.

(c)
Given $U\subseteq\mathbb R^n$, a function $f$ on $U$ is called \emph{analytic} if there exists an open superset $\tilde U$ of $U$ and an analytic function $\tilde f$ on $\tilde U$ such that $\tilde f\equiv f$ on $U$.

(d) Given $M\subseteq\mathbb R^n$, a flow on $M$ is a continuous function from $[T,\infty)$ to $M$ for some $T\in\mathbb R$.

Here is our main theorem.
\begin{theorem}
Let $U$ be an open domain in $\mathbb R^n$
and $U\subseteq\cup_i D_i$ be a finite closed cover such that each $D_i$ is a convex polytope. 
Suppose $M\ni0$ is an analytic submanifold in $U$ and $V(x)$ is a $C^1$ vector field on $U$ such that 
\begin{enumerate}[label=(\roman*)]
    \item $V(x)$ is analytic on each $D_i\cap U$,
    \item $DV(x)(T_xM)\subseteq T_xM$ for all $x\in M$,
    \item $V(0)=0$, and
    \item $DV(0)|_{T_0M}$ is diagonalizable and has only negative eigenvalues. 
\end{enumerate}
Suppose $x(t)$ is a flow on $M$ such that
\begin{enumerate}[label=(\roman*)]
    \item $x'(t)=V(x(t))$ on $[0,\infty)$, and 
    \item $\lim_{t\rightarrow\infty}x(t)=0$.
\end{enumerate}
Then there exists a sequence of finitely many increasing numbers $0=t_0<t_1<...<t_{n-1}<t_n=\infty$ such that for all $i\in\{1,...,n\}$, $x([t_{i-1},t_i))\subseteq D_j$ for some $j$.
\end{theorem}

To prove Theorem 1.1, by compactness we only need to prove Theorems 1.2 and 1.3, describing the local behavior and the limiting behavior respectively. 

\begin{theorem}
Let $U$ be an open domain in $\mathbb R^n$
and $U\subseteq\cup_i D_i$ be a finite closed cover such that each $D_i$ is a convex polytope. Suppose $V(x)$ is a $C^1$ vector field on $U$ and is analytic on each $D_i\cap U$. If $x'(t)=V(x(t))$ near $0$, then there exists $\epsilon>0$ such that $x([0,\epsilon))\subseteq D_i$ for some $i$.
\end{theorem}

\begin{theorem}
Let $U$ be an open domain in $\mathbb R^n$
and $U\subseteq\cup_i D_i$ be a finite closed cover such that each $D_i$ is a convex polytope. 
Let $M\ni0$ be an $m$-dim analytic submanifold in $U$. 
Suppose $V(x)$ is a $C^1$ vector field on $U$ such that 
\begin{enumerate}[label=(\roman*)]
    \item $V(x)$ is analytic on each $D_i\cap U$,
    \item $DV(x)(T_xM)\subseteq T_xM$ for all $x\in M$,
    \item $V(0)=0$, and
    \item $DV(0)|_{T_0M}$ is diagonalizable and has $m$ negative eigenvalues $\lambda_1\geq...\geq\lambda_m$. 
\end{enumerate}
Suppose $x(t)$ is a flow on $M$ such that
\begin{enumerate}[label=(\roman*)]
    \item $x'(t)=V(x(t))$ on $[0,\infty)$, and 
    \item $\lim_{t\rightarrow\infty}x(t)=0$.
\end{enumerate}
Then there exists $T>0$ such that $x([T,\infty))\subseteq D_j$ for some $j$.
\end{theorem}

\subsection{Finiteness of the edge flips in the  discrete Yamabe flow}
As a consequence of Theorem 1.1, we prove the finiteness of the number of edge flips in the discrete Yamabe flow introduced in
\cite{gu2018discrete}. One may refer \cite{gu2018discrete} for backgrounds.
Other related work on discrete conformality and geodesic triangulations could be found in 
\cite
{wu2022surface,
luo2023convergence,
wu2021fractional,
luo2023deformation,
dai2023rigidity,
wu2023computing,
izmestiev2023prescribed,
luo2021deformation2,
luo2019koebe,
wu2020convergence,
luo2021deformation,
luo2022discrete,
wu2015rigidity,
sun2015discrete,
gu2018discrete,
gu2018discrete2,
gu2019convergence
}.
Here we give a very brief explanation.

A discrete conformal class in \cite{gu2018discrete} is parametrized by $\mathbb R^n$ and has a natural finite cell decomposition $\mathbb R^n=\cup_i U_i$. A discrete Yamabe flow $u(t)$ is an integral curve of a $C^1$ vector field $F(u)$ on $\mathbb R^n$, satisfying that 

(a) $F(u)$ is analytic on each $U_i$,

(b) $F(u)\in\mathbf 1^\perp_n=\{(x_1,...,x_n)\in\mathbb R^n:x_1+...+x_n=0\}$ for all $u$, and

(c) $DF(u)$ is always symmetric and negative definite as a transformation on $\mathbf 1_n^\perp$.

(d) $F(\bar u)=0$ for some $\bar u\in u(0)+\mathbf 1_n^\perp$.
\\
A discrete Yamabe flow $u(t)$ always exists on $[0,\infty)$ and satisfies that 

(a) $u(t)-u(0)\in\mathbf 1_n^\perp$.

(b) $u'(t)=F(u(t))$, and

(c) $\lim_{t\rightarrow \infty}u(t)=\bar u$.
\\
Furthermore, the analytic change of coordinates $(u_1,...,u_n)\mapsto(e^{-2u_1},...,e^{-2u_n})$ will map each cell $U_i$ to $D_i\cap\mathbb R^n_{>0}$ for some convex polytope $D_i$. 

If we let $M$ be the image of $u(0)+\mathbf 1_n^\perp$ under this coordinate change, then Theorem 1.1 implies the finiteness of the number of switches in the new coordinates.

\subsection{Notations and a convex geometric estimate}
$ $

(a) Given $x\in\R^n$ (or $\Z^n$), denote $x=(x_1,...,x_n)$.

(b) Given $x\in\R^n$ (or $\Z^n$), denote $|x|=|x|_1=|x_1|+...+|x_n|$.

(c) Given $x\in\R^n$ (or $\Z^n$), denote $|x|_2=\sqrt{x_1^2+...+x_n^2}$.

(d) Given $x\in\R^n$ (or $\Z^n$), denote $|x|_\infty=\max_i|x_i|$.

(e) Given $x\in\R^n, I\in\Z^n$, denote $x^I=x_1^{I_1}... x_n^{I_n}$.

(f) $\mathbf 1_n$ denotes $(1,...,1)\in\mathbb R^n$.

(g) Given $U\subseteq\mathbb R^n$ and $p\in\mathbb R^n$, denote $d(p,U)=\inf_{q\in U}|p-q|_2$.

(h) Given a half space $H\subseteq\mathbb R^n$, $d_s(x, H)$ denotes the singed distance from $x$ to $H$. To be specific, $d_s(\cdot,H)$ is a linear function on $\mathbb R^n$ such that
$d_s(x,H)=d(x,H)$ whenever $d(x,H)>0$.

\begin{lemma}Suppose $D=\cap_{i=1}^mH_i$ is a convex polytope in $\mathbb R^n$, and $H_i$'s are half spaces. Then for all $p\in int(D)$ and $x\notin D$ we have that 
$$
\frac{\min_i d(p,\partial H_i)}{|x-p|_2}\cdot d(x,D)\leq \max_i d(x,H_i)\leq d(x,D).
$$
\end{lemma}

\begin{proof}
The second part of the inequality is obvious.
Denote $q$ as the intersection of $\partial D$ and the straight arc from $p$ to $x$. Assume $q\in \partial H_j$ and then
$$
\frac{|x-q|_2}{|x-p|_2}=\frac{d(x,H_j)}{d(p,\partial H_j)+d(x,H_j)}
\leq
\frac{d(x,H_j)}{d(p,\partial H_j)}
\leq
\frac{\max_i d(x,H_i)}{\min_i d(p,\partial H_i)}.
$$
Notice that $|x-q|_2\geq d(x,D)$ and we are done.
\end{proof}

\section{Local finiteness}

\begin{proof}[Proof of Theorem 1.2]
Without loss of generality, assume $x(0)=0$. 
It suffices to show that for any $D_i$ there exists $\epsilon>0$ such that 
$x([0,\epsilon])\subseteq D_i$ or $x([0,\epsilon])\cap D_i=\{0\}$.

Fix $D=D_i$.
Let $V_*$ be an analytic vector field on an open superset of $D\cap U$ such that $V_*(x)=V(x)$ on $D\cap U$. Let $y(t)$ be the local analytic solution to 
$$
\left\{
\begin{array}{ll}
 y(0)=0\\
 y'(t)= V_*(y(t))
\end{array}
\right..
$$
If there exists $\epsilon>0$ such that $y([0,\epsilon])\subseteq D$, then by Picard's uniqueness theorem $x(t)= y(t)\in D$ on $[0,\epsilon]$.

So we may assume that for all $\epsilon>0$, $y([0,\epsilon])\not\subseteq D$. Suppose $D=\cap_{j=1}^m H_j$ where $H_j$'s are half spaces. 
For a fixed $j$, $d(y(t),H_j)=0$ for all small $t$ or $d(y(t),H_j)=at^k+o(t^k)$ for some $a>0$ and $k\in\mathbb Z_{\geq1}$. Since we cannot have that $d(y(t),H_j)=0$ for all small $t$ and for all $j$, 
we have that 
$$
\max_{j}d(y(t),H_j)=at^k+o(t^k)
$$ 
for some $a>0$ and $k\in\mathbb Z_{\geq1}$. By Lemma 1.4, there exists constants $a_1,a_2>0$ such that
$a_1 t^k\leq d(y(t),D)\leq a_2 t^k$ for small $t$.

We will prove $x(t)=y(t)+o(t^k)$, and then $x(t)\notin D$ for small $t>0$.
We will prove $x(t)=y(t)+o(t^j)$ inductively for all $j=0,1,...,k$. Obviously $x(t)=y(t)+o(1)$. Now we assume $x(t)=y(t)+o(t^j)$ and $j\leq k-1$. 
Denote $y_*(t)$ as the closest point to $y(t)$ in $D$, and then $ y_*(t)= y(t)+o(t^j)=x(t)+o(t^j)$ and
$$
V(x(t))
=V( y_*(t))+o(t^j)
=V_*( y_*(t))+o(t^j)
$$
$$
=V_*(y(t))+o(t^j)=y'(t)+o(t^j)
$$
and
$$
x(t)=\int_0^tV(x(s))ds=y(t)+o(t^{j+1}).
$$
\end{proof}
\section{$\lambda$-series and convergence}
\subsection{$n$-dim analytic functions}

An $n$-dim analytic function $f(x)$ near $0$ in $\mathbb R^m$ can be written as
$$
f(x)=\sum_{I\in\mathbb Z^m_{\geq0}}b_Ix^I.
$$
where $b_I\in\mathbb R^n$ and 
$$
|b_I|\leq M^{|I|+1}
$$ 
for some constant $M>0$.

Suppose $f(x)=\sum_Ib_Ix^I$ is an $n$-dim analytic function near $0$ in $\mathbb R^m$, and $g(x)=\sum_Jc_Jx^J$ is an $m$-dim analytic function near $0$ in $\mathbb R^k$ with $c_0=0$. Then $(f\circ g)(x)$ is an $n$-dim function near $0$ in $\mathbb R^k$, defined as
$$
(f\circ g)(x)=f(g(x)).
$$
Furthermore, $f\circ g$ is analytic near $0$ and
$$
(f\circ g)(x)=f(g(x))=\sum_Ib_I(\sum_Jc_Jx^J)^I
=\sum_Jd_Jx^J
$$
where
\begin{equation}
\label{bj}
d_J=
\sum_{I}b_I\sum_{J_{i,j}:\sum_{i=1}^m\sum_{j=1}^{I_i}J_{i,j}=J}\prod_{i,j}(c_{J_{i,j}})_i.
\end{equation}
Here we let $d_0=b_0$ as a convention and denote $(c_{J_{i,j}})_i$ as the $i$-th component of $c_{J_{i,j}}$.
Notice that the summation in equation (3.1) is well-defined since it contains only finitely many nonzero terms. This is because $c_0=0$ and if $|I|>|J|$ then $\sum_{i=1}^m\sum_{j=1}^{I_i}J_{i,j}=J$ forces some $J_{i,j}$ to be $0$.

\subsection{$\lambda$-series}
$P(t)=\sum_{i=0}^k a_it^i$ is called an \emph{$n$-dim polynomial} if $a_i\in\mathbb R^n$ for all $i$. 
Given $\lambda\in\mathbb R^m_{<0}$, we define an \emph{$n$-dim $\lambda$-series} as a formal expression
$$
x(t)=x(t;\lambda)=\sum_{J\in\mathbb Z^m_{\geq0}}P_J(t)e^{\lambda\cdot Jt}
$$
where each $P_J(t)$ is an $n$-dim polynomial.
To be rigorous, such an $n$-dim $\lambda$-series could be represented by a map $J\mapsto P_J$ from $\mathbb Z_{\geq0}^m$ to the space of $n$-dim polynomials. However, we will always represent a $\lambda$-series as the above infinite summation for better intuition.
The \emph{formal derivative} of a $\lambda$-series is also a $\lambda$-series naturally defined by
$$
(\sum_{J\in\mathbb Z^m_{\geq0}}P_J(t)e^{\lambda\cdot Jt})'
=\sum_{J\in\mathbb Z^m_{\geq0}}(P_J'(t)+\lambda\cdot JP_J(t))e^{\lambda\cdot Jt}.
$$
If $f(x)=\sum_{I\in\mathbb Z^n_{\geq0}}b_Ix^I$ is an analytic function in $n$ variables and
$$
x(t)=\sum_{J\in\mathbb Z^m_{\geq0}}P_J(t)e^{\lambda\cdot Jt}
$$
is an $n$-dim $\lambda$-series with $P_0(t)=0$, we can heuristically formally expand $f(x(t))$ as 
$$
f(x(t))=
\sum_{I\in\mathbb Z^n_{\geq0}}b_{I}\bigg(\sum_{J\in\mathbb Z^m_{\geq0}}P_J(t)e^{\lambda\cdot Jt}\bigg)^I
=\sum_{J\in\mathbb Z^m_{\geq0}}Q_J(t)e^{\lambda\cdot Jt}
$$
 where
\begin{equation}
\label{expansion}
Q_J(t)=\sum_{I\in\mathbb Z_{\geq0}^n}b_I
\sum_{J_{i,j}:
\sum_{i=1}^n\sum_{j=1}^{I_i}J_{i,j}=J
}
\prod_{i,j}(P_{J_{i,j}})_i(t).
\end{equation}
Here we let $Q_0=b_0$ as a convention and denote
$(P_{J_{i,j}})_i(t)$ as the $i$-th component of the $n$-dim polynomial $P_{J_{i,j}}(t)$.
So given such analytic function $f(x)$ and $\lambda$-series $x(t)$, we define 
$
(f\circ x)(t)$ as a formal $\lambda$-series 
$$
(f\circ x)(t)=\sum_JQ_J(t)e^{\lambda\cdot Jt}
$$ 
where $Q(t)$ is defined as in equation (3.2).
Notice that the summation in equation (3.2) is well-defined since it contains only finitely many nonzero terms. This is because $P_0(t)=0$ and if 
$|I|>|J|$ then $\sum_{i=1}^n\sum_{j=1}^{I_i}J_{i,j}=J$ forces some $J_{i,j}$ to be $0$.
Also notice that equation (3.2) is similar to equation (3.1). This is because a $\lambda$-series could be viewed as a power series with $m$ variables $e^{\lambda_1t},...,e^{\lambda_mt}$ and polynomial-valued coefficients.
\begin{prop}
\label{composition}
    Suppose $f(x)=\sum_{I}b_Ix^I$ is a $k$-dim analytic function in $n$ variables, $g(x)=\sum_{I}c_Ix^I$ is an $n$-dim analytic function in $m$ variables with $c_0=0$, and $x(t)=\sum_{I}P_I(t)e^{\lambda\cdot Jt}$ is an $m$-dim $\lambda$-series with $P_0=0$. Then
$$
(f\circ (g\circ x))(t)=((f\circ g)\circ x)(t).
$$
\end{prop}
\begin{proof}
This can be shown by a routine but tedious computation, using equations (3.1) and (3.2).
\end{proof}

\subsection{Convergence of a $\lambda$-series}

We have a simple sufficient condition for a $\lambda$-series to converge.
\begin{prop}
Given $\lambda\in\mathbb R^m_{<0}$ and a
$\lambda$-series $x(t)=\sum_J P_J(t)e^{\lambda\cdot Jt}$, suppose there exists $T>0$ and $q\in\mathbb R_{>0}$ such that $|P_J(t)|\leq t^{q|J|}$ for all $J\in\mathbb Z^m_{\geq0}$ and $t\geq T$. 
Then 

(a) the series $x(t)$ converges absolutely and uniformly for sufficiently large $t$, and

(b) for any $a\in\mathbb R$, as a real-valued function
$$
x(t)=\sum_{J\in\mathbb Z^m_{\geq0}:\lambda\cdot J\geq a}P_J(t)e^{\lambda\cdot Jt}+o(e^{at})
$$
as  $t\rightarrow\infty$.
\end{prop}
\begin{proof}
(a) 
Without loss of generality, we may assume that $T$ is so large that $t^qe^{\lambda_it}$ is decreasing and less than $1$ on $[T,\infty)$ for all $i=1,...,m$.

For $t\in[T,\infty)$
$$
\sum_{|J|_\infty>n}|P_J(t)e^{\lambda\cdot Jt}|
\leq
\sum_{|J|_\infty>n}t^{q|J|}e^{\lambda\cdot Jt}
\leq
\sum_{|J|_\infty>n}T^{q|J|}e^{\lambda\cdot JT}
$$
$$
=\sum_{J}T^{q|J|}e^{\lambda\cdot JT}
-\sum_{|J|_\infty\leq n}T^{q|J|}e^{\lambda\cdot JT}
$$
$$
=\prod_{i=1}^m\left(\sum_{j=0}^\infty (T^{q}e^{\lambda_iT})^j\right)
-\prod_{i=1}^m\left(\sum_{j=0}^n 
(T^{q}e^{\lambda_iT})^j\right)
$$
which is independent on $t$ and goes to $0$ as $n$ goes to infinity, since for all $i$
$$
\sum_{j=0}^n 
(T^{k}e^{\lambda_iT})^j
\rightarrow
\sum_{j=0}^\infty 
(T^{q}e^{\lambda_iT})^j
=\frac{1}{1-T^qe^{\lambda_i T}}<\infty
$$
as $n\rightarrow\infty$.

(b)
Pick a large $n$ such that $\lambda\cdot J<a$ for all $J$ with $|J|_\infty> n$. Then

$$
\sum_{\lambda\cdot J<a}P_J(t)e^{\lambda\cdot Jt}
=
\sum_{|J|_\infty> n}P_J(t)e^{\lambda\cdot Jt}
+o(e^{at}).
$$
Then we finish the proof by showing that for sufficiently large $t$
$$
\sum_{|J|_\infty> n}P_J(t)e^{\lambda\cdot Jt}
\leq\sum_{|J|_\infty>n}t^{q|J|}e^{\lambda\cdot Jt}
$$
$$
=\sum_{J}t^{q|J|}e^{\lambda\cdot Jt}
-\sum_{|J|_\infty\leq n}t^{q|J|}e^{\lambda\cdot Jt}
$$
$$
=\prod_{i=1}^m\left(\sum_{j=0}^\infty (t^{q}e^{\lambda_it})^j\right)
-\prod_{i=1}^m\left(\sum_{j=0}^n 
(t^{q}e^{\lambda_it})^j\right)
$$
$$
=\sum_{i=1}^mO((t^qe^{\lambda_it})^{n+1})
=o(e^{at}).
$$
\end{proof}

\begin{prop}
\label{analytic_lambda}
Given $\lambda\in\mathbb R^m_{<0}$ and an $n$-dim
$\lambda$-series $x(t)=\sum_J P_J(t)e^{\lambda\cdot Jt}$ with $P_0(t)=0$, suppose there exists 
$T>1$ and $q\in\mathbb R_{>0}$ such that $|P_J(t)|\leq t^{q|J|}$ for all $J\in\mathbb Z^m_{\geq0}$ and $t\geq T$. 
    If $f(x)=\sum_Ib_Ix^I$ is a $1$-dim analytic function on a neighborhood of $0$ in $\mathbb R^n$, then there exists $r>0$ such that the $\lambda$-series
    $$
    (f\circ x)(t)=\sum_{J}Q_J(t)e^{\lambda\cdot Jt}.
    $$
    satisfies that $|Q_J(t)|\leq t^{r|J|}$ for all $J$ and $t\geq T$.
Furthermore, for sufficiently large $t$,
\begin{equation}
({f\circ x})(t)=f(x(t)).    
\end{equation}
Here $(f\circ x)(t)$ and $x(t)$ are viewed as two functions induced by taking the limit of the corresponding $\lambda$-series.
Equation (3.3) means that the limit of the $\lambda$-series $(f\circ x)(t)$ is equal to the analytic function $f$ evaluated at the limit of the $\lambda$-series $x(t)$.

\end{prop}
\begin{proof}
    Without loss of generality, assume $f(0)=0$. Suppose $t\geq T$ and by equation (3.2),
$$
|Q_J(t)|
\leq \tilde Q_J(t):=\sum_{I\in\mathbb Z_{\geq0}^n}|b_I|
\sum_{J_{i,j}:
\sum_{i=1}^n\sum_{j=1}^{I_i}J_{i,j}=J
}
\prod_{i,j}|(P_{J_{i,j}})_i(t)|
$$
$$
=\sum_{I\in\mathbb Z_{\geq0}^n}|b_I|
\sum_{J_{i,j}\neq0:
\sum_{i=1}^n\sum_{j=1}^{I_i}J_{i,j}=J
}
\prod_{i,j}|(P_{J_{i,j}})_i(t)|
$$
$$
\leq\sum_{I\in\mathbb Z_{\geq0}^n}|b_I|
\sum_{J_{i,j}\neq0:
\sum_{i=1}^n\sum_{j=1}^{I_i}J_{i,j}=J
}
\prod_{i,j}t^{q|J_{i,j}|}
$$
$$
=t^{q|J|}\sum_{I\in\mathbb Z_{\geq0}^n}|b_I|
\sum_{J_{i,j}\neq0:
\sum_{i=1}^n\sum_{j=1}^{I_i}J_{i,j}=J
}
1
$$
where
$$
c_J:=\sum_{I\in\mathbb Z_{\geq0}^n}|b_I|
\sum_{J_{i,j}\neq0:
\sum_{i=1}^n\sum_{j=1}^{I_i}J_{i,j}=J
}
1
$$
is the coefficient of $x^J$ in the local analytic function
$$
\sum_I |b_I|(\mathbf 1_n\sum_{i=1}^\infty x^i)^I=f_*(\frac{x}{1-x}\mathbf 1_n)
$$
where $f_*(x)=\sum_{I}|b_I|x^I$ is analytic near $0$.

So $c_J\leq R^{|J|}$ for some $R>0$, and for all $J\in\mathbb Z^m_{\geq0}$ and $t\geq T$
$$
|Q_J(t)|\leq \tilde Q_J(t)
\leq R^{|J|}t^{q|J|}\leq t^{(q+\log_T R)|J|}
.
$$

Suppose $t$ is sufficiently large so that 
$x(t)$ and $(f\circ x)(t)$ and $\sum_J\tilde Q_Je^{\lambda\cdot Jt}$ all converge absolutely and $f$ is analytic near $x(t)$.
To prove $({f\circ x})(t)=f( x(t))$ we need to show that
$$
\sum_JQ_J(t)e^{\lambda\cdot Jt}=\sum_{I}b_I(\sum_JP_J(t)e^{\lambda\cdot Jt})^I.
$$
Let 
$$
A_{N,M}=\sum_{|I|\leq N}b_I(\sum_{J:|J|\leq M}P_{J}(t)e^{\lambda\cdot Jt})^I
$$
and 
$$
B_N=\sum_{J:|J|\leq N}Q_J(t)e^{\lambda\cdot Jt}.
$$
If $N\leq M$, then
$$
|A_{N,M}-B_N|
=\left|\sum_{J:|J|>N}e^{\lambda\cdot Jt}\sum_{I:|I|\leq N}b_I
\sum_{J_{i,j}\neq0:
\sum_{i=1}^n\sum_{j=1}^{I_i}J_{i,j}=J,|J_{i,j}|\leq M
}
\prod_{i,j}(P_{J_{i,j}})_i(t)\right|
$$
$$
\leq\sum_{J:|J|>N}\tilde Q_J(t)e^{\lambda\cdot Jt}
$$
which is independent on $M$ and goes to $0$ as $N$ goes to infinity. Let $M\rightarrow\infty$ and then $N\rightarrow\infty$, and then we get 
$$
|\sum_{I}b_I(\sum_JP_J(t)e^{\lambda\cdot Jt})^I
-\sum_JQ_J(t)e^{\lambda\cdot Jt}|=0.
$$
\end{proof}

\section{Limiting behavior for analytic vector fields}
\subsection{Formal $\lambda$-series solutions to ODEs}
    Let $V(x)=\sum_I{b_I}x^I$ be an analytic vector field on a neighborhood of $0$ in $\mathbb R^m$ such that $V(0)=0$ and 
    $$
    \Lambda:=DV(0)=\text{diag}(\lambda_1,...,\lambda_m)
    $$
    for some $\lambda=(\lambda_1,...,\lambda_m)\in\mathbb R^m_{<0}$.
A $\lambda$-series 
$$
x(t)=\sum_JP_J(t)e^{\lambda\cdot Jt}
$$ 
with $P_0(t)=0$ formally satisfies $x'=V(x)$ if and only if for all $J\neq0$,
$$
P_J'+\lambda\cdot J P_J
=\sum_{I}b_I\sum_{J_{i,j}:\sum_{i=1}^n\sum_{j=1}^{I_i}J_{i,j}=J}\prod_{i,j}(P_{J_{i,j}})_i,
$$
which is equivalent to that
\begin{equation}
\left(\frac{d}{dt}-(\Lambda-\lambda\cdot J)\right)
P_J=Q_J
\end{equation}
where
\begin{equation}
Q_J=\sum_{I:|I|\geq2}b_I
\sum_{J_{i,j}:
\sum_{i=1}^m\sum_{j=1}^{I_i}J_{i,j}=J
}
\prod_{i,j}(P_{J_{i,j}})_i.
\end{equation}

\subsection{Construction of $\lambda$-series solutions}

Given $u\in\mathbb R$ and a 1-dim polynomial $Q(t)$, define
$$
(\frac{d}{dt}-u)^{-1}Q(t)=
-u^{-1}(1+u^{-1}\frac{d}{dt}
+u^{-2}\frac{d^2}{dt^2}+...)Q(t),
$$
if $u\neq0$ and  
$$
(\frac{d}{dt}-u)^{-1}Q(t)=\int_0^tQ(s)ds
$$
if $u=0$.
It is straightforward to verify that 
$$
(\frac{d}{dt}-u)
(\frac{d}{dt}-u)^{-1}Q(t)=Q(t).
$$
Given an $m$-dim diagonal matrix $U=\text{diag}(u_1,...,u_m)$ and an $m$-dim polynomial $Q(t)$, 
$$
P(t)=(\frac{d}{dt}-U)^{-1}Q(t)
$$ 
is defined to be such that
$$
P_i(t)=(\frac{d}{dt}-u_i)^{-1}Q_i(t)
$$
for all $i=1,...,m$. Clearly we have that 
$$
(\frac{d}{dt}-U)(\frac{d}{dt}-U)^{-1}Q(t)=Q(t).
$$

   Denote 
   $$
   \vec e_1=(1,0,...,0)\in\mathbb R^m,
   $$
   $$
   \vdots
   $$
   $$
   \vec e_m=(0,...,0,1)\in\mathbb R^m.
   $$
Given any $c\in\mathbb R^m$, we construct a formal $\lambda$-series solution
$$
x(t;c)=\sum_JP_J(t;c)e^{\lambda\cdot Jt}
$$
to $x'=V(x)$ as follows.

\begin{equation}
    P_0(t;c)=0,
\end{equation}
\begin{equation}
    P_{\vec e_i}(t;c)=c_i\vec e_i,
\end{equation}
\begin{equation}
P_J(t;c)=\left(\frac{d}{dt}-(\Lambda-\lambda\cdot J)\right)^{-1}Q_J(t;c),
\end{equation}
if $|J|\geq2$
where
\begin{equation}
Q_J(t;c)=\sum_{I\in\mathbb Z_{\geq0}^m:|I|\geq2}b_I
\sum_{J_{k,l}:
\sum_{i=1}^m\sum_{j=1}^{I_i}J_{i,j}=J
}
\prod_{i,j}(P_{J_{i,j}})_i(t;c).
\end{equation}
It is easy to verify such $x(t;c)$'s are formal $\lambda$-series solutions to $x'=V(x)$ parameterized by $c\in\mathbb R^m$.

\subsection{Dominating functions}
Given a $1$-dim polynomial in one variable
$$
P(t)=\sum_{i=0}^qa_it^i,
$$
we denote 
$$
P^*(t)
=\left(\sum_{i=0}^q a_it^i\right)^*
=\sum_{i=0}^q |a_i t^i|
$$
as a $1$-dim function in one variable.
Given an $n$-dim polynomial in one variable 
$P(t)=(P_1(t),...,P_n(t))$, then
denote 
$$
\deg(P)=\max_i\deg(P_i)
$$
and
$$
P^*(t)=\max_{i}P_i^*(t).
$$
It is routine to verify the following property.
\begin{prop}
Suppose $t\in\mathbb R$ and $a\in\mathbb R$ and $P,\tilde P$ are two $n$-dim polynomials in one variable. Then

    (a) $(P+\tilde P)^*(t)\leq P^*(t)+\tilde P^*(t)$,

    (b) $(aP(t))^*=|a|P^*(t)$, and

    (c) $(P(t)\tilde P(t))^*\leq P^*(t)\tilde P^*(t)$ if $n=1$.
\end{prop}
\begin{lemma}
Suppose $u\geq2$ and $Q(t)$ are $1$-dim polynomial and 
$$
P(t)=(\frac{d}{dt}-u)^{-1}Q(t).
$$    
Then
$$
P^*(t)\leq Q^*(t)
\quad\text{
 for all 
}
\quad
t\geq \frac{2\deg(Q)}{u}.
$$ 
\end{lemma}

\begin{proof}

Assume $Q(t)=a_qt^q+...+a_1t+a_0$
and denote
$$
\mathcal A=(\frac{d}{dt}-u)^{-1}
=-u^{-1}(1+u^{-1}\frac{d}{dt}
+u^{-2}\frac{d^2}{dt^2}+...).
$$
If $0\leq i\leq q$ and $t\geq{2q}/{u}\geq 2i/u$,
$$
(\mathcal A(t^i))^*
=u^{-1}t^i(1+\frac{i}{u t}+\frac{i(i-1)}{(u t)^2}+...
)
\leq\frac{1}{2}t^i(1+\frac{1}{2}+\frac{1}{2^2}+...)\leq t^i
$$
and then
$$
P^*(t)=\left(\sum_{i=0}^qa_i\mathcal A(t^i)\right)^*
\leq\sum_{i=0}^q|a_i|(\mathcal A(t^i))^*\leq\sum_{i=1}^q|a_i|t^i=Q^*(t).
$$
\end{proof}
We have a similar result for $n$-dim polynomials as a direct consequence.
\begin{corollary}
    Suppose $U=\text{diag}(u_1,...,u_n)$ is a diagonal matrix such that $u_i\geq2$ for all $i$.
    If $P(t),Q(t)$ are $n$-dim polynomials such that
$$
(\frac{d}{dt}-U)P(t)=Q(t),
$$    
then 
$$
P^*(t)\leq Q^*(t)
\quad\text{
 for all 
}
\quad
t\geq \frac{2\deg(Q)}{\min_i u_i}.
$$ 
\end{corollary}

\subsection{Convergence of $\lambda$-series solutions}
\begin{lemma}
\label{lemma_a_J}
Let $n\in\mathbb Z_{>0}$ and $M>0$ and $a_J\in\mathbb R$ for all $J\in\mathbb Z_{\geq0}^m$. $a_J$ is defined inductively as follows.
\begin{align}
 &  a_0=0   \\
 &  a_J=1  &\text{if $|J|=1$}\\
 &  a_J=\sum_{I\in\mathbb Z^n_{\geq0}:|I|\geq2}M^{|I|}
\sum_{J_{i,j}:\sum_{i=1}^n\sum_{j=1}^{I_i}J_{i,j}=J}
\prod_{i,j}a_{J_{i,j}}
&\text{ if $|J|\geq2$.}
\end{align}
Then there exists $R>0$ such that 
$|a_J|\leq R^{|J|}$.
\end{lemma}
\begin{proof}
Consider the following analytic function $F(x_1,...,x_m,y)$ of $(m+1)$ variables.
$$
F(x_1,...,x_m,y)=\frac{1}{(1-My)^n}-(Mn+1)y+(x_1+...+x_m).
$$
We have that $F(0)=1$ and $F'_y(0)=-1\neq0$. Then by the analytic implicit function theorem, there exists an analytic function $f(x)=\sum_{J}b_Jx^J$ near $0\in\mathbb R^m$ such that $f(0)=0$ 
and 
$$
F(x_1,...,x_m,f(x_1,...,x_m))=1,
$$
i.e.,
\begin{equation}
\label{Mn}
1-(x_1+...+x_m)+(Mn+1)f=\frac{1}{(1-Mf)^n}
=(1+Mf+(Mf)^2+...)^n.
\end{equation}

It suffices to show that $b_J=a_J$ for all $J$.
Clearly we have that $b_0=0=a_0$ and $b_J=1=a_J$ for all $J$ with $|J|=1$.
Then it suffices to show that
$$
b_J=\sum_{I\in\mathbb Z^n_{\geq0}:|I|\geq2}M^{|I|}
\sum_{J_{i,j}:\sum_{i=1}^n\sum_{j=1}^{I_i}J_{i,j}=J}
\prod_{i,j}b_{J_{i,j}}
$$
for all $J$ with $|J|\geq2$. 
Now we fix $J\in\mathbb Z^m_{\geq0}$ with $|J|\geq2$. The coefficient of $x^J$ on the left-hand side of equation (\ref{Mn}) is 
$$
(Mn+1)b_J.
$$
Denote $C_J(P)$ as the coefficient of $x^J$ in the power series $P$. 
Then the coefficient of $x^{J}$ on the right-hand side of equation (\ref{Mn}) is 
$$
\sum_{I\in\mathbb Z^n_{\geq1}}
C_{J}\left(\prod_{i=1}^n (Mf)^{I_i}\right)
=
\sum_{I\in\mathbb Z^n_{\geq1}}M^{|I|}
C_{J}\left(\prod_{i=1}^n f^{I_i}\right)
$$
where
$$
C_{J}\left(\prod_{i=1}^n f^{I_i}\right)=
C_{J}\left(\prod_{i=1}^n (\sum_{J}b_{J}x^{J})^{I_i}\right)=
\sum_{J_{i,j}:\sum_{i=1}^n\sum_{j=1}^{I_i}J_{i,j}=J}\prod_{i,j}b_{J_{i,j}}.
$$
It remains to show that
$$
Mnb_J=
\sum_{I\in\mathbb Z^n_{\geq0}:|I|=1}M^{|I|}
\sum_{J_{i,j}:\sum_{i=1}^n\sum_{j=1}^{I_i}J_{i,j}=J}
\prod_{i,j}b_{J_{i,j}}.
$$
This is a consequence of that for all $I$ with $|I|=1$
$$
\sum_{J_{i,j}:\sum_{i=1}^n\sum_{j=1}^{I_i}J_{i,j}=J}
\prod_{i,j}b_{J_{i,j}}=b_J.
$$
\end{proof}

\begin{theorem}
\label{limit behavior of analytic field}
    Let $U\ni0$ be an open set in $\mathbb R^m$ and $V$ be an analytic vector field on $U$ such that $V(0)=0$ and 
    $$
    DV(0)=\text{diag}(\lambda_1,...,\lambda_m)
    $$
    for some $\lambda=(\lambda_1,...,\lambda_m)\in\mathbb R^m_{<0}$.
    Suppose the $n$-dim $\lambda$-series
    $$
    x(t)=\sum_{J\in\mathbb Z^m_{\geq0}}P_J(t)e^{\lambda\cdot Jt}
    $$
    formally satisfies that $P_0=0$ and $x'(t)=V(x(t))$ as $\lambda$-series.
    Then we have the following.
    
    (a) There exists $T>0$ and $r>0$ such that
    $|P_J(t)|\leq t^{r|J|}$ for all $J\in\mathbb Z^m_{\geq0}$ and $t\geq T$.

    (b) $x(t),x'(t)$ converge and satisfy $x'(t)=V(x(t))$ as functions for sufficiently large $t$. 

    (c) If $\lambda_0>\max_i\lambda_i$ and $C\in\mathbb R^m$ is nonzero and $x(t;c)$ denotes the formal solution defined as in Section 4.2, as $t\rightarrow\infty$ we have
    $$
x(t;c+C)-x(t;c)=\sum_{i=1}^mC_ie^{\lambda_it}\vec e_i+o(e^{(\lambda_C+\lambda_0)t})
$$
where $\lambda_C=\max\{\lambda_i:i\in\{1,...,m\},C_i\neq0\}$.
\end{theorem}
\begin{proof}
(a) 

Without loss of generality, we may assume 
$
0>\lambda_1\geq...\geq\lambda_m.
$
Denote $\Lambda=\text{diag}(\lambda_1,...,\lambda_m)=DV(0)$ and suppose 
$$
V(x)=\Lambda x+\sum_{I\in\mathbb Z^m_{\geq0}:|I|\geq2}b_Ix^I.
$$
Since formally $x'(t)=V(x(t))$ we have that
$$
P_J'(t)+\lambda\cdot JP_J(t)
=\Lambda P_J(t)+Q_J(t)
$$
where
$$
Q_J(t)=\sum_{I\in\mathbb Z_{\geq0}^m:|I|\geq2}b_I
\sum_{J_{k,l}:
\sum_{k=1}^m\sum_{l=1}^{I_k}J_{k,l}=J
}
\prod_{k,l}(P_{J_{k,l}})_k(t).
$$
So
$$
Q_J^*(t)\leq\sum_{I\in\mathbb Z_{\geq0}^m:|I|\geq2}|b_I|
\sum_{J_{k,l}:
\sum_{k=1}^m\sum_{l=1}^{I_k}J_{k,l}=J
}
\prod_{k,l}(P^*_{J_{k,l}})(t).
$$
and
\begin{equation}
\left(\frac{d}{dt}-(\Lambda-\lambda\cdot J)\right)
P_J=Q_J
\end{equation}
Here $\Lambda-\lambda\cdot J$ is a diagonal matrix whose smallest diagonal entry is 
$$
(\lambda_m-\lambda\cdot J).
$$
Pick $s\in\mathbb Z_{\geq0}$ sufficiently large such that for all $J$ with $|J|\geq s$, 
$$
\lambda_m-\lambda\cdot J\geq\frac{-\lambda_1 |J|}{2}\geq2.
$$
Then for any $J$ with $|J|\geq s$,
$\deg (P_J)=\deg(Q_J)$.

Let $p\in\mathbb R_{\geq0}$ be a constant such that 
$$
\deg (P_J)\leq p|J|\quad\text{ and }\quad\deg(Q_J)\leq p|J|
$$
for all $J\in\mathbb Z^m_{\geq0}$ with $|J|< s$.
Then by induction it is straightforward to see that 
$$
\deg (P_J)\leq p|J|\quad\text{ and }\quad\deg(Q_J)\leq p|J|
$$
for all $J$.
By Corollary 4.3, for $J\in\mathbb Z^m_{\geq0}$ with $|J|\geq s$ and 
$$
t
\geq\frac{4p}{-\lambda_1}
=\frac{2p|J|}{-\lambda_1|J|/2}
\geq\frac{2\deg(Q_J)}{\lambda_m-\lambda\cdot J}
$$ 
we have 
$P_J^*(t)\leq Q_J^*(t)$.

Pick a sufficiently large $q$ such that
$$
P_J^*(t)\leq t^{q|J|}
$$
for all $t\geq2$ and $J\in\mathbb Z^m_{\geq0}$ with $|J|\leq s$. 
Suppose $|b_I|\leq M^{|I|}$ for some $M>0$ and all $I\in\mathbb Z^m_{\geq0}$.
Let $a_J$ be defined as in Lemma 4.4.
By induction it is straightforward to see that $P_J^*(t)\leq a_Jt^{q|J|}$ for all 
$$
t\geq\max\{\frac{4p}{-\lambda_1},2\}
$$ 
and $J\in\mathbb Z^m_{\geq0}$.
By Lemma \ref{lemma_a_J} there exists $R>0$ such that
$a_J\leq R^{|J|}$. So 
$$
|P_J(t)|_\infty\leq P^*_J(t)\leq R^{|J|}t^{q|J|}
\leq 2^{|J|\log_2 R}\cdot t^{q|J|}
\leq t^{(q+\log_2R)|J|}
$$
for all $J$ and 
$$
t\geq\max\{\frac{4p}{-\lambda_1},2\}.
$$ 

(b) By Propositions 3.2 (a) and 3.3, $x(t),V(x(t))$ converge uniformly. Since any partial sum of $x'(t)=V(x(t))=(V\circ x)(t)$ is the derivative of the corresponding partial sum of $x(t)$ as functions, $x'(t)=V(x(t))$ is the derivative of $x(t)$ as functions.

(c) By Proposition 3.2 (b), we only need to compare the coefficients of $e^{\lambda\cdot Jt}$ for nonzero $J$ satisfying $\lambda\cdot J\geq\lambda_C+\lambda_0$. Such $J$ could either be $\vec e_i$ or satisfy that $J_i=0$ if $C_i\neq0$. The first case is immediate from the definition. In the second case of $J$, we can straightforwardly show that
$P_J(t;c)=P_J(t,c+C)$ by induction.
\end{proof}

\section{Limiting behavior for piecewise analytic vector fields}
\subsection{Preliminary estimates}
\begin{lemma}
Let $U\ni0$ be an open domain in $\mathbb R^n$ and $U\subset \cup_iD_i$ be a closed finite cover. Suppose $V(x)$ is a $C^1$ vector field on $U$ such that $V(x)$ is analytic on each $D_i\cap U$. Then 
$$
V(x)-V(y)=DV(0)(x-y)+O((|x|+|y|)|x-y|).
$$
\begin{proof}
On a small neighborhood $U_0$ of $0$, $DV(x)$ is  Lipschitz continuous on each $D_i\cap U_0$, since $V(x)$ is analytic on $D_i\cap U_0$. Since $DV(x)$ is continuous on $U_0$, it is easy to see that $DV(x)$ is Lipschitz continuous on $U_0$. So on $U_0$ we may assume that
$$
|DV(x)-DV(0)|\leq C|x|.
$$
for some constant $C>0$.
Then
$$
V(x)-V(y)-DV(0)(x-y)
$$
$$
=\left(\int_0^1DV(tx+(1-t)y)(x-y)dt\right)-DV(0)(x-y)  
$$
$$
=\int_0^1\bigg(DV(tx+(1-t)y)-DV(0)\bigg)(x-y)dt 
$$
$$
\leq\left(\int_0^1C|(tx+(1-t)y)|_2dt\right)
|x-y|_2
$$
$$
=O((|x|+|y|)|x-y|).
$$
\end{proof}
\end{lemma}

\begin{lemma}
    Suppose $V(x)$ is a $C^1$ vector field on a neighborhood $U$ of $0$ in $\mathbb R^m$, such that 
    $$
    V(x)=\Lambda x+O(|x|^2)
    $$
    where 
    $$
    \Lambda=\text{diag}(\lambda_1,...,\lambda_m)
    $$
    and
    $$
    0>\lambda_1\geq...\geq\lambda_m.
    $$
    If $x'(t)=V(x(t))$ on $[0,\infty)$ and $x(t)=o(1)$ as $t\rightarrow\infty$, then
    $x(t)=O(e^{\lambda_1t})$ as $t\rightarrow\infty$.
\end{lemma}
\begin{proof} As $t\rightarrow\infty$,
    $$
    (|x|_2^2)'=2x\cdot x'=2x\cdot V(x)
    $$
    $$
    =2x\cdot (\Lambda x+O(|x|^2))
    \leq 2\lambda_1|x|_2^2+O(|x|^3).
    $$
    If $\lambda_0>\lambda_1$,
    $$
    (e^{-2\lambda_0t}|x|_2^2)'=
    -2\lambda_0e^{-2\lambda_0t}|x|_2^2+(2\lambda_1|x|_2^2+O(|x|^3))e^{-2\lambda_0t}
    $$
    $$
    \leq(2\lambda_1-2\lambda_0)e^{-2\lambda_0t}|x|_2^2+e^{-2\lambda_0t}O(|x|)^3.
    $$
So $(e^{-2\lambda_0t}|x|_2^2)'<0$ if $t$ is sufficiently large, and then 
$$
e^{-2\lambda_0t}|x|_2^2=O(1)
$$
and
$$
|x|=O(e^{\lambda_0t}).
$$

So 
    $$
    (e^{-2\lambda_1t}|x|_2^2)'=
    -2\lambda_1e^{-2\lambda_1t}|x|_2^2+(2\lambda_1|x|_2^2+O(|x|^3))e^{-2\lambda_1t}
    $$
    
    $$
    =e^{-2\lambda_1t}O(|x|)^3
    =O(e^{(3\lambda_0-2\lambda_1)t}).
    $$
So if $\lambda_0=3\lambda_1/4>\lambda_1$
$$
e^{-2\lambda_1t}|x|_2^2
=|x(0)|_2^2+\int_0^tO(e^{\lambda_1s/4})ds
=|x(0)|_2^2+O(1)=O(1)
$$
and $|x|=O(e^{\lambda_1t})$.
\end{proof}

\begin{lemma}
    Suppose $V(x)$ is a $C^1$ vector field on a neighborhood $U$ of $0$ in $\mathbb R^m$, such that 
    $$
    DV(0)=
    \text{diag}(\lambda_1,...,\lambda_m)
    $$
    where
    $$
    0>\lambda_1\geq...\geq\lambda_m.
    $$
    Suppose 
    $$
    x'(t)=V(x(t)),\quad y'(t)=V(x(t))
    $$ 
    on $[0,\infty)$ and 
    $$
    x(t)=o(1),\quad y(t)=o(1)
    $$ 
    as $t\rightarrow\infty$. 
    If $y(t)-x(t)=O(e^{\lambda_0 t})$ for some $\lambda_0<\lambda_m$, then
    $x(t)=y(t)$.
\end{lemma}
\begin{proof}
Let $f(s)=sy+(1-s)x$. Then
    $$
    V(y)-V(x)=V(f(1))-V(f(0))
    =\int_0^1(V\circ f)'(s)ds
    $$
    and
    $$
    |V(y)-V(x)|_2\leq\int_0^1|(V\circ f)'(s)|_2ds
    $$
    $$
    =\int_0^1|DV(f(s))|_2\cdot|f'(s)|_2ds
    =\max_{0\leq s\leq1}|DV(f(s))|_2\cdot|y-x|_2
    \leq \frac{\lambda_0+\lambda_m}{2}|y-x|_2
    $$
    if $|x|,|y|$ are sufficiently small.
    So for $t$ sufficiently large,
    $$
    (|y-x|_2^2)'=2(y-x)\cdot (V(y)-V(x))
    \geq-2|y-x|_2\cdot|V(y)-V(x)|_2\geq(\lambda_0+\lambda_m)|y-x|_2^2.
    $$
    So $|y-x|_2^2\geq e^{(\lambda_0+\lambda_m)t}|y(0)-x(0)|_2^2$. So we must have $y(0)=x(0)$ if $y(t)-x(t)=O(e^{\lambda_0 t})$.
    
\end{proof}
\begin{lemma}
Suppose $a<0,b<0,a\neq b$, and $x(t),h(t)$ are two functions on $[T,\infty)$ for some $T\in\mathbb R$, such that 
$$
x'(t)=ax(t)+h(t)
$$ 
and 
$$
h(t)=o(e^{bt})
$$ 
as $t\rightarrow\infty$. Then
$$
x(t)=Ce^{at}+o(e^{bt})
$$ 
as $t\rightarrow\infty$ for some constant $C$.
\end{lemma}
\begin{proof}
If $a<b$,
$$
x(t)=e^{at}\int_T^t
e^{-as}h(s)ds+Ce^{at}
$$
$$
=e^{at}\int_T^t
e^{-as}\cdot o(e^{bs})ds+Ce^{at}
$$
$$
=e^{at}\int_T^t
o(e^{(b-a)s})ds+Ce^{at}
$$
$$
=e^{at}\cdot
o(e^{(b-a)t})+Ce^{at}
$$
$$
=
o(e^{bt})+Ce^{at}
$$
for some constant $C$.

If $a>b$,
$$
x(t)=e^{at}\int_\infty^t
e^{-as}h(s)ds+Ce^{at}
$$
$$
=e^{at}\int_\infty^t
e^{-as}\cdot o(e^{bs})ds+Ce^{at}
$$
$$
=e^{at}\int_\infty^t
o(e^{(b-a)s})ds+Ce^{at}
$$
$$
=e^{at}\cdot
o(e^{(b-a)t})+Ce^{at}
$$
$$
=
o(e^{bt})+Ce^{at}
$$
for some constant $C$.

\end{proof}

\subsection{Proof of Theorem 1.3}
\begin{proof}[Proof of Theorem 1.3]

It suffices to prove that for any cell $D$ in the finite cover, $x(t)\in D$ for sufficiently large $t$ or $x(t)\notin D$ for sufficiently large $t$. Let us fix a cell $D$ and assume $0\in D$ without loss of generality.

    By a linear transformation, we may assume that $T_0M$ is tangent to the $(x_1...x_m)$-plane and
    $$
\Lambda:=DV(0)|_{T_0M}=\text{diag}(\lambda_1,...,\lambda_m).
$$
Without loss of generality, we may assume that
$$
M=\{(x_1,...,x_n):(x_1,...,x_m)\in U_0,(x_{m+1},...,x_n)=f(x_1,...,x_m)\}
$$
for some open set $U_0\subset\mathbb R^m$ and analytic function $f$ on $U_0$.
Then $f(0)=0$ and $Df(0)=0$.
Denote
$$
\bar x(t)=(x_1(t),...,x_m(t)),\text{ and}
$$ 
$$
\bar V(x)=(V_1(x),...,V_m(x)).
$$
Then 
$$
\bar x'(t)
=\bar V(\bar x,f(\bar x))=:\tilde V(\bar x).
$$
where
$\tilde V(\bar x)$ is $C^1$ in a neighborhood
of $0$ in $\mathbb R^m$. Furthermore,
$$
D\tilde V(0)=
\left(\Lambda,\frac{\partial \bar V}{\partial (x_{m+1},...,x_n)}(0)\right)
(Id_m,Df(0))^T
=\Lambda.
$$
So by Lemma 5.2 $\bar x(t)=O(e^{\lambda_1 t})$ and then $x(t)=O(e^{\lambda_1t})$. 

Suppose 
$$
p_m(x_1,...,x_n)=(x_1,...,x_m)
$$ 
is the projection from $\mathbb R^n$ to $\mathbb R^m$. Then $\tilde V(\bar x)$ is analytic on $p_m(M\cap D_i)$ for each cell $D_i$ in the finite cover.
Then by Lemma 5.1,
\begin{equation}
\tilde V(\bar x)-\tilde V(\bar y)=\Lambda(\bar x-\bar y)+O((|x|+|y|)|x-y|).
\end{equation}

Let $V_*$ be the analytic vector field on an open superset of $D\cap U$ such that $V_*(x)=V(x)$ on $D\cap U$. Denote
$$
\bar V_*(x)=(V_{*,1}(x),...,V_{*,m}(x))
$$
and
$$
\tilde V_*(\bar x)=\bar V_*(\bar x,f(\bar x)).
$$
Given $c\in\mathbb R^m$, let $\bar y(t;c)$ be the formal $\lambda$-series solution to $\bar y'=\tilde V_*(\bar y)$ as constructed in Section 4.2.

Pick $\lambda_0\in(\lambda_1,0)$ such that 
$
k\lambda_0\neq\lambda_i
$
for all $k\in\mathbb Z_{\geq0}$ and $i=1,...,m$.
Let
$$
A=\{k\in\mathbb Z_{\geq0}:\text{there exists $c\in\mathbb R^m$ such that } \bar x(t)-\bar y(t;c)=o(e^{k\lambda_0t})\text{ as $t\rightarrow\infty$}\}
$$
$A$ is nonempty since $0\in A$. Suppose $D=\cap_iH_i$ where $H_i$'s are half spaces, and denote $y(t;c)=(\bar y(t;c),f(\bar y(t;c)))$ as a $n$-dim $\lambda$-series.

(a) If $A$ has no maximum, then we can find $k\in A$ and $c\in\mathbb R^m$ such that $k\lambda_0<\lambda_m$ and $\bar x(t)-\bar y(t;c)=o(e^{k\lambda_0t})$. Then by Lemma 5.3 $\bar x(t)=\bar y(t;c)$ as $\mathbb R^m$-valued functions. Then by Proposition 3.3 $x(t)=y(t;c)$ as $\mathbb R^n$-valued functions. Then for each $H_i$, the signed distance
$d_s(y(t;c),H_i)$ is a $\lambda$-series and converge for sufficiently large $t$. 
Furthermore, by Proposition 3.3 the limit of the $\lambda$-series $d_s(y(t;c),H_i)$ is indeed the signed distance from $y(t;c)$ to $H_i$.
For each $H_i$,
$d(y(t;c),H_i)\equiv 0$ for sufficiently large $t$ or  
$d(y(t;c),H_i)
=at^qe^{rt}+o(at^qe^{rt})$ for some $a>0$, $q\in\mathbb Z_{\geq0}$ and $r<0$. 
So
$x(t)=y(t;c)\notin D$ for sufficiently large $t$ or $x(t)=y(t;c)\in D$ for sufficiently large $t$.

(b)
If $k=\max A$, let $c$ be such that 
$$
\bar x(t)=\bar y(t;c)+o(e^{k\lambda_0t}).
$$
If  
$$
d( y(t;c),H_i)=at^qe^{r t}+o(at^qe^{r t})
$$ 
for some $H_i$ and $a>0$ and $q\in\mathbb Z_{\geq0}$ and $r\geq k\lambda_0$,
then 
$$
d( x(t),H_i)=at^qe^{r t}+o(at^qe^{r t}).
$$
If not, $d(y(t;c),H_i)=o(e^{k
\lambda_0 t})$ for all $H_i$, and then 
$$
d(y(t;c),D)\leq d(y(t;c),0)=O(e^{\lambda_1t}) 
$$
and by Lemma 1.4 
$$
d(y(t;c),D)=o(e^{k\lambda_0 t}).
$$
Denote $y_*(t)$ as the closest point to $y(t;c)$ in $D$.
Then 
$$
y_*(t)=O(y(t))=O(e^{\lambda_1t})
$$ 
and
$$
x(t)=y(t;c)+o(e^{k\lambda_0t})=y_*(t)+o(e^{k\lambda_0t}),
$$
and by Lemmas 5.1 and 5.2
$$
(\bar x-\bar y)'
=\tilde V(\bar x)-\tilde V_*(\bar y)
$$
$$
=(\tilde V(\bar x)-\tilde V(\bar y_*))
+(\tilde V_*(\bar y_*)-\tilde V_*(\bar y))
$$
$$
=\Lambda(\bar x-\bar y_*)+O((|\bar x|+|\bar y_*|)|\bar x-\bar y_*|)
+\Lambda(\bar y_*-\bar y)+O((|\bar y_*|+|\bar y|)|\bar y_*-\bar y|)
$$
$$
=\Lambda(\bar x-\bar y)+O(e^{\lambda_1 t}\cdot e^{k\lambda_0t})
=\Lambda(\bar x-\bar y)+o(e^{(k+1)\lambda_0t}).
$$
Then by Lemma 5.4 for all $i=1,...,m$, 
$$
\bar x_i(t)-\bar y_i(t;c)=C_ie^{\lambda_it}+o(e^{(k+1)\lambda_0t})
$$
for some constant $C_i$. 
Let $C=(C_1,...,C_m)$, and then $C$ is nonzero by the maximality of $k$. 
Since $\bar x(t)-\bar y(t;c)=o(e^{k\lambda_0t})$, $C_i=0$ if $\lambda_i>k\lambda_0$.
So
$$
\lambda_C:=\max\{\lambda_i:i\in\{1,...,m\},C_i\neq0\}<k\lambda_0.
$$
By Theorem 4.5 (c),
$$
\bar y(t;c+C)=\bar y(t;c)+\sum_{i=1}^mC_ie^{\lambda_it}\vec e_i+o(e^{(\lambda_C+\lambda_0)t})
$$
$$
=\bar x(t;c)+o(e^{(\lambda_C+\lambda_0)t})
=\bar x(t;c)+o(e^{(k+1)\lambda_0t}).
$$ 
This contradicts the maximality of $k$.
\end{proof}

\bibliography{sreference}

\providecommand{\bysame}{\leavevmode\hbox to3em{\hrulefill}\thinspace}
\providecommand{\MR}{\relax\ifhmode\unskip\space\fi MR }
\providecommand{\MRhref}[2]{%
  \href{http://www.ams.org/mathscinet-getitem?mr=#1}{#2}
}
\providecommand{\href}[2]{#2}
\begin{thebibliography}{10}

\bibitem{dai2023rigidity}
Song Dai and Tianqi Wu, \emph{Rigidity of the delaunay triangulations of the
  plane}, arXiv preprint arXiv:2305.02609 (2023).

\bibitem{gu2019convergence}
David Gu, Feng Luo, and Tianqi Wu, \emph{Convergence of discrete conformal
  geometry and computation of uniformization maps}, Asian Journal of
  Mathematics \textbf{23} (2019), no.~1, 21--34.

\bibitem{gu2018discrete2}
Xianfeng Gu, Ren Guo, Feng Luo, Jian Sun, Tianqi Wu, et~al., \emph{A discrete
  uniformization theorem for polyhedral surfaces ii}, Journal of Differential
  Geometry \textbf{109} (2018), no.~3, 431--466.

\bibitem{gu2018discrete}
Xianfeng~David Gu, Feng Luo, Jian Sun, and Tianqi Wu, \emph{A discrete
  uniformization theorem for polyhedral surfaces}, Journal of Differential
  Geometry \textbf{109} (2018), no.~2, 223--256.

\bibitem{izmestiev2023prescribed}
Ivan Izmestiev, Roman Prosanov, and Tianqi Wu, \emph{Prescribed curvature
  problem for discrete conformality on convex spherical cone-metrics}, arXiv
  preprint arXiv:2303.11068 (2023).

\bibitem{luo2022discrete}
Feng Luo, Jian Sun, and Tianqi Wu, \emph{Discrete conformal geometry of
  polyhedral surfaces and its convergence}, Geometry \& Topology \textbf{26}
  (2022), no.~3, 937--987.

\bibitem{luo2019koebe}
Feng Luo and Tianqi Wu, \emph{Koebe conjecture and the weyl problem for convex
  surfaces in hyperbolic 3-space}, arXiv preprint arXiv:1910.08001 (2019).

\bibitem{luo2021deformation}
Yanwen Luo, Tianqi Wu, and Xiaoping Zhu, \emph{The deformation space of
  geodesic triangulations and generalized tutte's embedding theorem}, arXiv
  preprint arXiv:2105.00612 (2021).

\bibitem{luo2021deformation2}
\bysame, \emph{The deformation spaces of geodesic triangulations of flat tori},
  arXiv preprint arXiv:2107.05159 (2021).

\bibitem{luo2023convergence}
\bysame, \emph{The convergence of discrete uniformizations for genus zero
  surfaces}, Discrete \& Computational Geometry (2023), 1--24.

\bibitem{luo2023deformation}
\bysame, \emph{The deformation space of delaunay triangulations of the sphere},
  Pacific Journal of Mathematics \textbf{323} (2023), no.~1, 115--127.

\bibitem{sun2015discrete}
Jian Sun, Tianqi Wu, Xianfeng Gu, and Feng Luo, \emph{Discrete conformal
  deformation: algorithm and experiments}, SIAM Journal on Imaging Sciences
  \textbf{8} (2015), no.~3, 1421--1456.

\bibitem{wu2015rigidity}
Tianqi Wu, Xianfeng Gu, and Jian Sun, \emph{Rigidity of infinite hexagonal
  triangulation of the plane}, Transactions of the American Mathematical
  Society \textbf{367} (2015), no.~9, 6539--6555.

\bibitem{wu2021fractional}
Tianqi Wu and Xu~Xu, \emph{Fractional combinatorial calabi flow on surfaces},
  arXiv preprint arXiv:2107.14102 (2021).

\bibitem{wu2023computing}
Tianqi Wu and Shing-Tung Yau, \emph{Computing harmonic maps and conformal maps
  on point clouds.}, Journal of Computational Mathematics \textbf{41} (2023),
  no.~5.

\bibitem{wu2020convergence}
Tianqi Wu and Xiaoping Zhu, \emph{The convergence of discrete uniformizations
  for closed surfaces}, arXiv preprint arXiv:2008.06744 (2020).

\bibitem{wu2022surface}
Yingying Wu, Tianqi Wu, and Shing-Tung Yau, \emph{Surface eigenvalues with
  lattice-based approximation in comparison with analytical solution}, arXiv
  preprint arXiv:2203.03603 (2022).

\end{thebibliography}
\bibliographystyle{amsplain}

\end{document}